\documentclass[10pt,reqno]{amsart}
\usepackage{latexsym,amsmath,amssymb,amscd}
\usepackage[all]{xy}
\usepackage{enumerate}

\def\today{\ifcase \month \or
   January \or February \or March \or April \or
   May \or June \or July \or August \or
   September \or October \or November \or December \fi
   \space\number\day , \number\year}
   
\makeatletter
  \newcommand\@dotsep{4.5}
  \def\@tocline#1#2#3#4#5#6#7{\relax
     \ifnum #1>\c@tocdepth 
     \else
     \par \addpenalty\@secpenalty\addvspace{#2}%
     \begingroup \hyphenpenalty\@M
     \@ifempty{#4}{%
     \@tempdima\csname r@tocindent\number#1\endcsname\relax
        }{%
         \@tempdima#4\relax
           }%
      \parindent\z@ \leftskip#3\relax \advance\leftskip\@tempdima\relax
      \rightskip\@pnumwidth plus1em \parfillskip-\@pnumwidth
       #5\leavevmode\hskip-\@tempdima #6\relax
       \leaders\hbox{$\m@th
       \mkern \@dotsep mu\hbox{.}\mkern \@dotsep mu$}\hfill
       \hbox to\@pnumwidth{\@tocpagenum{#7}}\par
       \nobreak
        \endgroup
         \fi}
\makeatother 

\begin{document}


\makeatletter
\@addtoreset{figure}{section}
\def\thefigure{\thesection.\@arabic\c@figure}
\def\fps@figure{h,t}
\@addtoreset{table}{bsection}

\def\thetable{\thesection.\@arabic\c@table}
\def\fps@table{h, t}
\@addtoreset{equation}{section}
\def\theequation{
\arabic{equation}}
\makeatother

\newcommand{\bfi}{\bfseries\itshape}

\newtheorem{theorem}{Theorem}
\newtheorem{acknowledgment}[theorem]{Acknowledgment}
\newtheorem{corollary}[theorem]{Corollary}
\newtheorem{definition}[theorem]{Definition}
\newtheorem{example}[theorem]{Example}
\newtheorem{lemma}[theorem]{Lemma}
\newtheorem{notation}[theorem]{Notation}
\newtheorem{problem}[theorem]{Problem}
\newtheorem{proposition}[theorem]{Proposition}
\newtheorem{question}[theorem]{Question}
\newtheorem{remark}[theorem]{Remark}
\newtheorem{setting}[theorem]{Setting}

\numberwithin{theorem}{section}
\numberwithin{equation}{section}

\renewcommand{\1}{{\bf 1}}
\newcommand{\Ad}{{\rm Ad}}
\newcommand{\Aut}{{\rm Aut}\,}
\newcommand{\ad}{{\rm ad}}
\newcommand{\alg}{{\rm alg}}
\newcommand{\botimes}{\bar{\otimes}}
\newcommand{\Ci}{{\mathcal C}^\infty}
\newcommand{\de}{{\rm d}}
\newcommand{\ee}{{\rm e}}
\newcommand{\End}{{\rm End}\,}
\newcommand{\Fl}{{\rm Fl}}
\newcommand{\id}{{\rm id}}
\newcommand{\ie}{{\rm i}}
\newcommand{\jump}{{\rm jump}}
\newcommand{\GL}{{\rm GL}}
\newcommand{\Gr}{{\rm Gr}}
\newcommand{\Hom}{{\rm Hom}\,}
\newcommand{\Ind}{{\rm Ind}}
\newcommand{\JH}{{\rm JH}}
\newcommand{\Ker}{{\rm Ker}\,}
\newcommand{\pr}{{\rm pr}}
\newcommand{\Ran}{{\rm Ran}\,}
\newcommand{\rank}{{\rm rank}\,}
\renewcommand{\Re}{{\rm Re}\,}
\newcommand{\sa}{{\rm sa}}
\newcommand{\spa}{{\rm span}\,}
\newcommand{\tsr}{{\rm tsr}}
\newcommand{\Tr}{{\rm Tr}\,}

\newcommand{\CC}{{\mathbb C}}
\newcommand{\HH}{{\mathbb H}}
\newcommand{\RR}{{\mathbb R}}
\newcommand{\TT}{{\mathbb T}}

\newcommand{\Ac}{{\mathcal A}}
\newcommand{\Bc}{{\mathcal B}}
\newcommand{\Cc}{{\mathcal C}}
\newcommand{\Hc}{{\mathcal H}}
\newcommand{\Ic}{{\mathcal I}}
\newcommand{\Jc}{{\mathcal J}}
\newcommand{\Kc}{{\mathcal K}}
\newcommand{\Lc}{{\mathcal L}}
\renewcommand{\Mc}{{\mathcal M}}
\newcommand{\Nc}{{\mathcal N}}
\newcommand{\Oc}{{\mathcal O}}
\newcommand{\Pc}{{\mathcal P}}
\newcommand{\Tc}{{\mathcal T}}
\newcommand{\Vc}{{\mathcal V}}
\newcommand{\Xc}{{\mathcal X}}
\newcommand{\Yc}{{\mathcal Y}}
\newcommand{\Wc}{{\mathcal W}}

\newcommand{\Fg}{{\mathfrak F}}
\newcommand{\Gg}{{\mathfrak G}}
\newcommand{\Ug}{{\mathfrak U}}
\newcommand{\Vg}{{\mathfrak V}}
\newcommand{\Wg}{{\mathfrak W}}

\renewcommand{\gg}{{\mathfrak g}}
\newcommand{\hg}{{\mathfrak h}}
\newcommand{\kg}{{\mathfrak k}}
\renewcommand{\ng}{{\mathfrak n}}

\newcommand{\ZZ}{\mathbb Z}
\newcommand{\NN}{\mathbb N}

\makeatletter
\title[Nonlinear oblique projections]{Nonlinear oblique projections}
\author{Ingrid Belti\c t\u a and Daniel Belti\c t\u a}
\address{Institute of Mathematics ``Simion Stoilow'' 
of the Romanian Academy, 
P.O. Box 1-764, Bucharest, Romania}
\email{ingrid.beltita@gmail.com, Ingrid.Beltita@imar.ro}
\email{beltita@gmail.com, Daniel.Beltita@imar.ro}
\thanks{This work was supported by a grant of the Romanian National Authority for Scientific Research and
Innovation, CNCS--UEFISCDI, project number PN-II-RU-TE-2014-4-0370}
\date{\today}
\makeatother

\begin{abstract} 
We construct nonlinear oblique projections along subalgebras of  nilpotent Lie algebras in terms of the Baker-Campbell-Hausdorff multiplication. 
We prove that these nonlinear projections are real analytic on every Schubert cell of the Grassmann manifold whose points are the subalgebras of the nilpotent Lie algebra under consideration. 
\end{abstract}

\maketitle


\section{Introduction}

The Grassmann manifold of linear subspaces of a finite-dimensional vector space plays an important role 
in linear algebra, operator theory, and differential geometry. 
In particular, the study of oblique projections and operator ranges can be transparently conducted 
from the perspective of that manifold and of its infinite-dimensional versions, as one can see for instance in \cite{ACrM15},   \cite{ArCrGo13}, \cite{AnCrMb13}, and  \cite{CrM10}. 

The oblique projections are linear projection operators 
defined by decompositions of a vector space into a direct sum of two subspaces. 
In this paper we study a nonlinear version of the oblique projections, 
replacing the commutative vector addition of a vector space $\Ug$ by a more general noncommutative group structure defined by a polynomial map $\Ug\times\Ug\to\Ug$, $(X,Y)\mapsto X\cdot Y$ satisfying $(tX)\cdot(sX)=(t+s)X$ for all $t,s\in\RR$ and $X\in\Ug$. 
As we will recall below in Section~\ref{nilp}, such a group structure turns $\Ug$ into a nilpotent Lie group and coincides with 
the Baker-Campbell-Hausdorff multiplication defined by a uniquely determined Lie bracket on~$\Ug$.  
(This construction makes sense if $\Ug$ is a Banach space and it then leads to some interesting problems, as discussed for instance in \cite{BB15a}.)
In this setting, the role of the Grassmann manifold is held by the set $\Gr^{\alg}(\Ug)$ of all subalgebras, rather than the linear subspaces of~$\Ug$. 
The natural nonlinear oblique projections along subalgebras defined in this way has been proved to be an important tool in representation theory of Lie groups (see \cite{CG90} and \cite{FuLu15}). 

In the noncommutative framework outlined above, we study these generalized oblique projections along subalgebras of a nilpotent Lie algebra, 
and we establish their analyticity properties on suitable Schubert cells 
 (Theorem~\ref{of}). 
This is our main result here, and it was motivated by our recent research on the structure of $C^*$-algebras of nilpotent Lie groups. 
(See \cite{BB17} and \cite{BBL17}.) 
We will briefly explain this motivation toward the end of the present paper, which is organized as follows: 
In Section~\ref{lin} we discuss analyticity of linear oblique projections, using the Moore-Penrose inverse. 
Then, in Section~\ref{cells} we establish some properties of the Schubert stratification of the Grassmann manifold, for later use. 
In Section~\ref{nilp} we briefly recall nilpotent Lie groups and algebras, and finally, in Section~\ref{nonlin} we obtain our main result on analyticity of nonlinear oblique projections.

\subsection*{General notation} 
For any finite-dimensional real vector space $\Ug$ we denote by $\Bc(\Ug)$ its unital associative algebra of linear operators on~$\Ug$. 
If $\Ug$ is endowed with a scalar product, and thus $\Ug$ is a finite-dimensional real Hilbert space, we denote $\Pc(\Ug):=\{P\in\Bc(\Ug)\mid P=P^2=P^*\}$, 
which is well known to be a compact real analytic submanifold of the real vector space $\Bc(\Ug)$. 
For every linear subspace $\Wc\subseteq\Ug$ we denote by $P_{\Wg}\in\Pc(\Ug)$ the orthogonal projection of $\Ug$ onto~$\Wg$. 
The \emph{Grassmann manifold} of $\Ug$ is the set $\Gr(\Ug)$ of all linear subspaces of $\Ug$. 
The map $\Gr(\Ug)\to\Pc(\Ug)$, $\Wg\mapsto P_{\Wg}$, is a bijection, 
and we endow $\Gr(\Ug)$ with the structure of a real analytic manifold that makes that bijection into a real analytic diffeomorphism. 

For any integer $n\ge 1$ we denote by $\Pc_n$ the set of all subsets of $\{1,\dots,n\}$. 
For every $e\in\Pc_n$ we denote $\complement e:=\{1,\dots,n\}\setminus e$, 
and we write $e=\{j_1<\cdots<j_d\}$ if $e=\{j_1,\dots,j_d\}$ with  $j_1<\cdots<j_d$.

\section{Analyticity of linear oblique projections}\label{lin}

If $\Ug$ is a finite-dimensional real vector space with two subspaces $\Ug_1,\Ug_2\subseteq\Ug$ with $\Ug=\Ug_1\dotplus\Ug_2$, 
then the corresponding linear \emph{oblique projection of $\Ug$ onto $\Ug_2$ along $\Ug_1$} 
is the linear operator $E\colon \Ug\to\Ug$ defined by the conditions $\Ker E=\Ug_1$ and $Ew=w$ for every $w\in\Ug_2$. 

The next theorem gives the analytic dependence of the above operator $E$ with respect to $\Ug_1$ in a suitable open subset of the Grassmann manifold. 

\begin{theorem}\label{lin0}
	Let $\Ug$  be a finite-dimensional real vector space with a fixed linear subspace $\Ug_0\subseteq\Ug$. 
	We denote $\Gr_{\Ug_0}(\Ug):=\{\Wg\in\Gr(\Ug)\mid \Ug_0\dotplus\Wg=\Ug\}$ and for every $\Wg\in\Gr_{\Ug_0}(\Ug)$ 
	let $E(\Wg)\colon\Ug\to\Ug$ be the oblique projection of $\Ug$ onto $\Ug_0$ along~$\Wg$. 
	Then $\Gr_{\Ug_0}(\Ug)$ is an open subset of $\Gr(\Ug)$ and the mapping 
	$$E\colon \Gr_{\Ug_0}(\Ug)\to\Bc(\Ug),\quad \Wg\mapsto E(\Wg)$$
	is real analytic. 
\end{theorem}

The proof of this theorem is based on Moore-Penrose inverses (sometimes called pseudo-inverses in the literature),
so we will briefly recall that notion and its properties that are needed below. 

\begin{definition}
	\normalfont
	Let $\Hc$ be any finite-dimensional real Hilbert space. 
	For every operator $A\in\Bc(\Hc)$ its \emph{Moore-Penrose inverse} is the operator 
	$A^\dagger:=B\in\Bc(\Hc)$ that is uniquely determined by the equations 
	$$ABA=A,\ BAB=B,\ (AB)^*=AB,\ (BA)^*=BA$$
	(see for instance \cite[subsect. 5.5.4]{GvL96}). 
	Then $A^\dagger$ exists for every $A\in\Bc(\Hc)$.
\end{definition}

\begin{lemma}\label{lin1}
	Let $\Hc$ be any finite-dimensional real Hilbert space, 
	$\Xc,\Yc\in\Gr(\Hc)$ with $\Xc\dotplus\Yc=\Hc$ and define the linear operator $F\colon\Hc\to\Hc$ 
	as the oblique projection on $\Xc$ along~$\Yc$. 
	  Then $F=P_{\Xc}((\1-P_{\Yc})P_{\Xc})^\dagger (\1-P_{\Yc})$.  
 \end{lemma}

\begin{proof}
	It follows by \cite[Th. 1]{Ha13} that for any operators $X,Y_0\in\Bc(\Hc)$ with $\Ran X=\Xc$ 
	and $\Ran Y_0=\Yc^\perp$, one has $F=X(Y_0^*X)^\dagger Y_0^*$. 
	Hence for $X=P_{\Xc}$ and $Y_0=\1-P_{\Yc}$ we obtain the assertion. 
\end{proof}

\begin{lemma}\label{lin2}
	Let $\Hc$ be a finite-dimensional real Hilbert space, 
	$M$ be a real analytic manifold and assume that $A\colon M\to\Bc(\Hc)$ 
	is a real analytic map such that the function $\dim\Ker A(\cdot)$ is locally constant on~$M$. 
	Then the map $A(\cdot)^\dagger$ is real analytic. 
\end{lemma}

\begin{proof}
	See \cite[Cor. 3.5]{LR12}. 
\end{proof}

The following lemma is essentially known but we sketch its proof in a coordinate-free manner, which carries over to infinite-dimensional Hilbert spaces. 

\begin{lemma}\label{lin3}
	Let $\Hc$ be a finite-dimensional real Hilbert space, 
	with a fixed linear subspace $\Ug_0\subseteq\Hc$. 
	Then the set $\Gr_{\Ug_0}(\Hc)$ is an open subset of the 
	real analytic manifold $\Gr(\Hc)$ and the map 
	$$\chi\colon \Gr_{\Ug_0}(\Hc)\to\Bc(\Ug_0^\perp,\Ug_0),\quad 
	\chi(\Wg)=P_{\Ug_0}\circ(P_{\Ug_0^\perp}\vert_{\Wg})^{-1}$$
	is a real analytic diffeomorphism with its inverse 
	$$\chi^{-1}\colon \Bc(\Ug_0^\perp,\Ug_0)\to \Gr_{\Ug_0}(\Hc), \quad 
	\chi^{-1}(T)=\{v+Tv\mid v\in\Ug_0^\perp\}.$$
\end{lemma}

\begin{proof}
	It is straightforward to check that the maps mentioned above are inverse to each other. 
	The real analytic structure of $\Gr(\Hc)$ was defined at the end of the Introduction via the bijective map 
	$\Gr(\Hc)\to\Pc(\Hc)$, $\Wg\mapsto P_{\Wg}$, hence in order to prove that the map $\chi^{-1}$ is real analytic we must check that the map 
	$$\Bc(\Ug_0^\perp,\Ug_0)\to \Pc(\Hc), \quad T\mapsto P_{\chi^{-1}(T)}$$
	is real analytic. 
	Here $P_{\chi^{-1}(T)}$ is the orthogonal projection from $\Hc=\Ug_0^\perp\oplus\Ug_0$ onto the graph of the operator $T\colon\Ug_0^\perp\to\Ug_0$, hence one has 
	$$P_{\chi^{-1}(T)}=
	\begin{pmatrix}
	 (\1+T^*T)^{-1}  & T^*(\1+TT^*)^{-1}\\
	 T(\1+T^*T)^{-1} & TT^*(\1+TT^*)^{-1}
	\end{pmatrix}\in\Pc(\Hc).$$
	(See for instance \cite[Eq. (5)]{An15}.)  
	This shows that the map $T\mapsto P_{\chi^{-1}(T)}$ is real analytic, and we are done. 
\end{proof}

\begin{proof}[Proof of Theorem~\ref{lin0}]
	We fix a scalar product on~$\Ug$, which turns $\Ug$ into a finite-dimensional real Hilbert space. 
	Then the real analytic structure of $\Gr_{\Ug_0}(\Ug)$ is clarified by Lemma~\ref{lin3}. 
	
	For every $\Wg\in\Gr_{\Ug_0}(\Ug)$ we have by Lemma~\ref{lin1}, 
	\begin{equation}\label{lin0_proof_eq1}
	E(\Wg)= P_{\Ug_0}((\1-P_{\Wg})P_{\Ug_0})^\dagger (\1-P_{\Wg}).
	\end{equation}
	We now check that 
	\begin{equation}\label{lin0_proof_eq2}
	\Ker((\1-P_{\Wg})P_{\Ug_0})=\Ug_0^\perp.
	\end{equation}
	In fact, $\Ker((\1-P_{\Wg})P_{\Ug_0})\supseteq\Ker P_{\Ug_0}=\Ug_0^\perp$.
	Conversely, if $x\in\Ker((\1-P_{\Wg})P_{\Ug_0})$, then $P_{\Ug_0}x\in\Ker(\1-P_{\Wg})=\Wg$, 
	hence $P_{\Ug_0}x\in\Ug_0\cap\Wg=\{0\}$, and then $x\in\Ug_0^\perp$. 
	
	Now recall that the map $\Gr(\Ug)\to\Pc(\Ug)$, $\Wg\mapsto P_{\Wg}$, 
	is a real analytic diffeomorphism by the definition of the real analytic structure of  $\Gr(\Ug)$, 
	and $\Pc(\Ug)$
	is a compact real analytic submanifold of the vector space~$\Bc(\Ug)$. 
	It follows by \eqref{lin0_proof_eq1}--\eqref{lin0_proof_eq2} and Lemma~\ref{lin2} 
	that the map $E\colon \Gr_{\Ug_0}(\Ug)\to\Bc(\Ug)$ is a composition of real analytic maps, 
	and this completes the proof. 
\end{proof}

\section{Schubert cells in Grassmann manifolds}\label{cells}

Throughout this section we denote by $\Ug$ a real vector space with $n:=\dim\Ug<\infty$.

\begin{definition}
	\normalfont 
	For every $j\in\{0,\dots,n\}$ we define 
	$$\Gr(\Ug,j):=\{\Wg\in\Gr(\Ug)\mid\dim\Wg=j\}.$$ 
A \emph{complete flag} in $\Gr(\Ug)$ is a sequence 
	${\Fg_\bullet}=(\Fg_j)_{0\le j\le n}$ with $\Fg_j\in \Gr(\Ug,j)$ and 
	$\Fg_j\subset\Fg_{j+1}$ for $j=0,\dots,n-1$.  
	Hence 
	$${\Fg_\bullet}:\quad \{0\}=\Fg_0\subset\Fg_1\subset\cdots\subset\Fg_n=\Ug.$$
	The \emph{flag manifold} of~$\Ug$ is the set $\Fl(\Ug)$ of all complete flags in $\Gr(\Ug)$. 
\end{definition}

We now introduce the Schubert stratification of the Grassmann manifold with respect to a complete flag. 
The Schubert cells are usually defined using intersections rather than sums of subspaces. 
However, the following definition is more suitable for our purposes, is related to the so-called coarse stratification from representation theory of nilpotent Lie groups (used for instance in \cite{BBL17}) and is equivalent to the traditional definition, as it follows by Proposition~\ref{Gr1}\eqref{Gr1_item8} below. 

\begin{definition}
	\normalfont
	For $\Wg\in\Gr(\Ug)$, its set of \emph{jump indices}\index{jump indices} with respect to a complete flag ${\Fg_\bullet}\in\Fl(\Ug)$ is 
	$$\jump_{{\Fg_\bullet}}(\Wg):=\{j\in\{1,\dots,n\}\mid \Fg_j\not\subset\Wg+\Fg_{j-1}\}.$$
	For every $e\in\Pc_n$ we also define its corresponding \emph{Schubert cell}\index{Schubert cell}
	$$\Gr_{{\Fg_\bullet},e}(\Ug):=\{\Wg\in\Gr(\Ug)\mid\jump_{{\Fg_\bullet}}(\Wg)=e\}.$$
	It is well known that $\Gr_{{\Fg_\bullet},e}(\Ug)$ is a submanifold of $\Gr(\Ug)$. 
	(See for instance 
	\cite[Ch. III]{Sch68}
	.) 
	The complement $\complement e\in\Pc_n$ is called 
	the \emph{Schubert symbol}\index{Schubert symbol} of the Schubert cell $\Gr_{{\Fg_\bullet},e}(\Ug)$. 
	
	For every ${\Fg_\bullet}\in\Fl(\Ug)$ one has the disjoint union 
	$$\Gr(\Ug)=\bigsqcup_{e\in\Pc_n}\Gr_{{\Fg_\bullet},e}(\Ug).$$
\end{definition}

\begin{remark}\label{Gr0}
	\normalfont
	If ${\Fg_\bullet}\in\Fl(\Ug)$ and $\Wg\in\Gr(\Ug)$, then one has 
	$$\Wg\subseteq\Wg+\Fg_1\subseteq\cdots\subseteq\Wg+\Fg_{n-1}\subseteq\Ug.$$
	Here  
	$\dim((\Wg+\Fg_j)/(\Wg+\Fg_{j-1}))\le 1$ for $j=1,\dots,n$. 
	Indeed, if $X_j\in\Fg_j\setminus\Fg_{j-1}$, then $\Fg_j=\Fg_{j-1}\dotplus\RR X_j$, hence one has the canonical linear isomorphism  
	$$\begin{aligned}
	(\Wg+\Fg_j)/(\Wg+\Fg_{j-1})
	 &=((\Wg+\Fg_{j-1})+\RR X_j)/(\Wg+\Fg_{j-1}) \\
	 &\simeq\RR X_j/(\RR X_j\cap(\Wg+\Fg_{j-1})).
	\end{aligned}$$
	It then follows by Proposition~\ref{Gr1} that 
	$$\jump_{{\Fg_\bullet}}(\Wg)= 
	\{j\in\{1,\dots,n\}\mid \dim((\Wg+\Fg_j)/(\Wg+\Fg_{j-1}))=1\}$$ 
	and this explains why the above set is called the set of jump indices. 
\end{remark}

\begin{proposition}\label{Gr1}
	Let ${\Fg_\bullet}\in\Fl(\Ug)$ and    $X_j\in\Fg_j\setminus\Fg_{j-1}$ for $j=1,\dots,n$. 
	If $\Wg\in\Gr(\Ug)$, $e\in\Pc_n$, 
	and $\Ug_e:=\spa\{X_j\mid j\in e\}$, 
	then the following conditions are equivalent: 
	\begin{enumerate}[{\rm(i)}]
		\item\label{Gr1_item1} 
		$e=\jump_{{\Fg_\bullet}}(\Wg)$; 
		\item\label{Gr1_item2} 
		$e=\{j\in\{1,\dots,n\}\mid \Wg+\Fg_{j-1}\subsetneqq \Wg+\Fg_j\}$;
		\item\label{Gr1_item3} 
		$e= \{j\in\{1,\dots,n\}\mid \dim((\Wg+\Fg_j)/(\Wg+\Fg_{j-1}))=1\}$; 
		\item\label{Gr1_item4} 
		$e=\{j\in\{1,\dots,n\}\mid X_j\not\in \Wg+\Fg_{j-1}\}$; 
		\item\label{Gr1_item5}
		if $e=\{j_1<\cdots<j_d\}$, then for $r=0,\dots,d$ and $j_r\le j<j_{r+1}$ one has $\dim(\Wg+\Fg_{j})=r+\dim\Wg$, 
		where $j_0:=0$ and $j_{d+1}:=n+1$. 
		\item\label{Gr1_item5bis}
		in the notation of \eqref{Gr1_item5}, for $r=0,\dots,d$ and $j_r\le j<j_{r+1}$, the family of vectors $(X_{j_s})_{1\le s\le r}$ is a basis in $\Fg_j\mod\Fg_j\cap\Wg$. 
		\item\label{Gr1_item6} $\complement e=\{i\in\{1,\dots,n\}\mid \Wg\cap\Fg_{i-1}\subsetneqq \Wg\cap\Fg_i\}$;
		\item\label{Gr1_item7} $\complement e= 
		\{i\in\{1,\dots,n\}\mid \dim((\Wg\cap\Fg_i)/(\Wg\cap\Fg_{i-1}))=1\}$; 
		\item\label{Gr1_item8} if $\complement e= \{k_1<\cdots<k_{n-d}\}$, 
		then $\dim(\Wg\cap\Fg_{k_i})=i$ and $\dim(\Wg\cap\Fg_{k_i-1})=i-1$ for $i=1,\dots,n-d$.
	\end{enumerate}
	If these conditions are satisfied, 
	then $\vert e\vert=\dim(\Ug/\Wg)$ 
	and in addition: 
	\begin{enumerate}[{\rm(i)}]
		\setcounter{enumi}{9}
		\item\label{Gr1_item10} 
		One has the direct sum decomposition $\Wg\dotplus\Ug_e=\Ug$.
	 	\item\label{Gr1_item11} 
		The family $(X_j+\Wg)_{j\in e}$ is a basis in $\Ug/\Wg$.
		\item \label{Gr1_item12}
		For every $i\in \complement e$ one has 
		$\Wg+\Ug_{i-1}=\Wg\dotplus(\Ug_e\cap\Ug_{i-1})$. 
	\end{enumerate}
\end{proposition}

\begin{proof}
	The implications	$\eqref{Gr1_item4}\Leftrightarrow\eqref{Gr1_item1}\Leftrightarrow\eqref{Gr1_item2}\Leftarrow\eqref{Gr1_item3}$ are clear, and $\eqref{Gr1_item2}\Rightarrow\eqref{Gr1_item3}$ by Remark~\ref{Gr0}. 
	Moreover, $\eqref{Gr1_item3}\Leftrightarrow\eqref{Gr1_item5}$. 
	
	$\eqref{Gr1_item5}\Leftrightarrow\eqref{Gr1_item5bis}$: 
	One has 
	$$\dim(\Fg_j/(\Fg_j\cap\Wg))=\dim((\Fg_j+\Wg)/\Wg), $$
	hence clearly $\eqref{Gr1_item5}\Leftarrow\eqref{Gr1_item5bis}$. 
	For the converse implication, using again the above equality  we obtain that
	the dimension of  the space 
	$\Fg_j/(\Fg_j\cap\Wg)$ is $r$.
	Hence it suffices to check that  $(\Fg_j\cap\Wg)\cap\spa\{X_{j_s}\mid 1\le s\le r\}=\{0\}$. 
	If there exists a vector $\sum\limits_{1\le s\le r}a_sX_{j_s}\in\Fg_j\cap \Wg$ with $t:=\max\{s\in\{1,\dots,r\}\mid a_s\ne 0\}$, 
	then $X_{j_t}\in\Wg+\spa\{X_{j_s}\mid 1\le s<t\}\subseteq\Wg+\Fg_{t-1}$, which is a contradiction with $j_t\in e$, by~\eqref{Gr1_item4}. 
	
	$\eqref{Gr1_item3}\Leftrightarrow\eqref{Gr1_item7}$:
	For every $j\in\{1,\dots,n\}$ one has 
	$$\dim(\Wg+\Fg_j)+\dim(\Wg\cap\Fg_j)
	=\dim\Wg+\dim\Fg_j
	=(n-d)+j$$
	and similarly 
	$$\dim(\Wg+\Fg_{j-1})+\dim(\Wg\cap\Fg_{j-1})
	=(n-d)+(j-1).$$
	Subtracting these equalities we obtain 
	$$\dim((\Wg+\Fg_j)/(\Wg+\Fg_{j-1}))
	+\dim((\Wg\cap\Fg_j)/(\Wg\cap\Fg_{j-1}))=1$$
	and thus $\eqref{Gr1_item3}\Leftrightarrow\eqref{Gr1_item7}$. 
	
	It is also clear that $\eqref{Gr1_item6}\Leftrightarrow\eqref{Gr1_item7}$ 
	and $\eqref{Gr1_item5}\Leftrightarrow\eqref{Gr1_item8}$, 
	hence \eqref{Gr1_item1}--\eqref{Gr1_item8} are equivalent.

	We will now prove that \eqref{Gr1_item1}--\eqref{Gr1_item8} imply that \eqref{Gr1_item10} and \eqref{Gr1_item11} hold true, and then $\vert e\vert=\dim(\Ug/\Wg)$ and \eqref{Gr1_item11} follow at once. 
	
	By Remark~\ref{Gr0} and \eqref{Gr1_item3}, 
	\begin{equation}\label{decomp2_proof_eq1}
	j\in e
	\Leftrightarrow \dim((\Wg+\Fg_j)/(\Wg+\Fg_{j-1}))=1
	\Leftrightarrow \Wg+\Fg_j=(\Wg+\Fg_{j-1})\dotplus \RR X_j.
	\end{equation}
	It thus follows that if $a\in\{1,\dots,d\}$, then 
	the following implication holds true: 
	\begin{equation}\label{decomp2_proof_eq2}
	j_a\le j<j_{a+1}\implies \Wg+\Fg_{j_a}=\Wg+\Fg_j.
	\end{equation}
	Using the above facts \eqref{decomp2_proof_eq1}--\eqref{decomp2_proof_eq2} repeatedly, we obtain 
	\allowdisplaybreaks
	\begin{align}
	\Ug =\Wg+\Fg_{j_d}
	&=(\Wg+\Fg_{j_{d-1}})\dotplus\RR X_{j_d} \nonumber\\
	&=((\Wg+\Fg_{j_{d-2}})\dotplus\RR X_{j_{d-1}})\dotplus\RR X_{j_d} \nonumber\\
	&=\cdots \nonumber\\
	&=(\cdots(\Wg+\Fg_{j_1})\dotplus\cdots\dotplus\RR X_{j_{d-1}})\dotplus\RR X_{j_d} \nonumber\\
	&=\Wg+\Ug_e \nonumber
	\end{align}
	which concludes the proof of~\eqref{Gr1_item10}.
	
	For \eqref{Gr1_item10}, 
	if $i\in\complement e$, then there exists $a\in\{1,\dots,r\}$ with 
	$j_a<i<j_{a+1}$. 
	As above, by \eqref{decomp2_proof_eq1}--\eqref{decomp2_proof_eq2} we obtain 
	\allowdisplaybreaks
	\begin{align}
	\Wg+\Fg_{i-1} =\Wg+\Fg_{j_a}
	&=(\Wg+\Fg_{j_{a-1}})\dotplus\RR X_{j_a} \nonumber\\
	&=\cdots \nonumber\\
	&=\Wg\dotplus(\RR X_1\dotplus\cdots\dotplus\RR X_{j_a}) \nonumber\\
	&=\Wg+\Ug_{e\cap\{1,\dots,i-1\}} \nonumber\\
	&=\Wg+(\Ug_e\cap\Fg_{i-1}) \nonumber
	\end{align}
	where the latter equality follows by the elementary fact that 
	\begin{equation}\label{Gr1_proof_eq1}
	\Ug_{e_1}\cap\Ug_{e_2}=\Ug_{e_1\cap e_2} 
	\text{ if }\Ug_{e_k}:=\spa\{X_j\mid j\in e_k\}\text{ for }k=1,2
	\text{ and }e_1,e_2\in\Pc_n.
	\end{equation}
	This completes the proof. 
\end{proof}

\begin{example}\label{ex1}
	\normalfont
	Let ${\Fg_\bullet}\in\Fl(\Ug)$ and  $X_j\in\Fg_j\setminus\Fg_{j-1}$ for $j=1,\dots,n$. 
	Then for every $e\in\Pc_n$ the subspace $\Ug_e:=\spa\{X_j\mid j\in e\}$ satisfies $\jump_{\Fg}(\Ug_e)=\complement e$. 
	
This can be proved either by a direct argument, or by an application of Proposition~\ref{Gr1}\eqref{Gr1_item6} along with \eqref{Gr1_proof_eq1}. 
\end{example}

\begin{example}\label{ex_dim1}
	\normalfont
	Let ${\Fg_\bullet}\in\Fl(\Ug)$ and $\Wg\in\Gr(\Ug,1)$.  For every $X\in\Wg\setminus\{0\}$,  
	defining $j_0:=\min\{j\in\{1,\dots,n\}\mid X\in\Fg_j\}$  
	one has $\jump_{{\Fg_\bullet}}(\Wg)=\{1,\dots,n\}\setminus\{j_0\}$.
	
	In fact, the definition of $j_0$ is equivalent to $X\in\Fg_{j_0}\setminus\Fg_{j_0-1}$. 
	If $X_j\in\Fg_j\setminus\Fg_{j-1}$ for $j=1,\dots,n$, then by Proposition~\ref{Gr1}\eqref{Gr1_item4}, 
	$$\begin{aligned}
	\complement (\jump_{{\Fg_\bullet}}(\Wg))
	&=\{j\in\{1,\dots,n\}\mid X_j\in\RR X+\Fg_{j-1}\} \\
	&=\{j\in\{1,\dots,n\}\mid X\in\Fg_j\setminus\Fg_{j-1}\} 
	\end{aligned}$$
	and now the assertion follows directly. 
\end{example}

\begin{example}\label{ex_codim1}
	\normalfont
	Let ${\Fg_\bullet}\in\Fl(\Ug)$, $\Wg\in\Gr(\Ug,n-1)$,  
	and $j_0:=\max\{j\in\{1,\dots,n\}\mid \Fg_{j-1}\subseteq\Wg\}$.  
	Then $\jump_{{\Fg_\bullet}}(\Wg)=\{j_0\}$.
	
	In fact, since $\dim(\Ug/\Wg)=1$, one has $\vert\jump_{{\Fg_\bullet}}(\Wg)\vert=1$ by  Proposition~\ref{Gr1}, 
	hence it suffices to prove that $j_0\in\jump_{{\Fg_\bullet}}(\Wg)$. 
	By the definition of $j_0$ one has 
	$\Fg_{j_0-1}\subseteq\Wg$ and $\Fg_{j_0}\not\subseteq\Wg$, 
	hence $\Fg_{j_0}\not\subseteq\Wg+\Fg_{j_0-1}$, 
	which shows that $j_0\in\jump_{{\Fg_\bullet}}(\Wg)$. 
\end{example}

\subsection*{Bases parameterized by Schubert cells}

The final result of this section is Theorem~\ref{decomp4} which requires the following notation. 
This theorem will be used in the proof of Theorem~\ref{of} via its Corollary~\ref{decomp6}.

\begin{theorem}\label{decomp4}
	If $X_1,\dots,X_m$ is a basis of $\Ug$ and $e\subseteq\{1,\dots,m\}$, then for every $\Wg\in\Gr_{{\Fg_\bullet},e}(\Ug)$ 
	there exists a unique family of vectors $\beta(\Wg)=(Y_1,\dots,Y_m)\in\Ug^m $ satisfying the following conditions: 
	\begin{enumerate}[{\rm(i)}]
		\item\label{decomp4_item1} 
		For every $j\in e$ one has $Y_j=X_j$. 
		\item\label{decomp4_item2} 
		For every $i\in\complement e$ one has $Y_i-X_i\in\Ug_e\cap\Fg_{i-1}$. 
	\end{enumerate}
	Moreover, $\beta(\Wg)$ is a basis of $\Ug$ having the following properties: 
	\begin{enumerate}[{\rm(i)}]
		\setcounter{enumi}{2}
		\item\label{decomp4_item3} 
		One has $\Fg_j=\spa\{Y_1,\dots,Y_j\}$ for $j=1,\dots,m$. 
		\item\label{decomp4_item4} 
		The set $\{Y_i\mid i\in \complement e\}$ is a basis of $\Wg$. 
		\item\label{decomp4_item5}  The mapping 
		$\beta\colon \Gr_{{\Fg_\bullet},e}(\Ug)\to\Ug^m$, $\Wg\mapsto\beta(\Wg)$, extends to a real analytic mapping on the open subset $\Gr_{\Ug_e}(\Ug)$ of $\Gr(\Ug)$.
		\end{enumerate}
\end{theorem}

The proof of the above statement requires the following lemma. 

\begin{lemma}\label{decomp3}
	If $Y_1,\dots,Y_m\in\Ug$ and $Y_j-X_j\in\Fg_{j-1}$ for $j=1,\dots,m$, 
	then $Y_1,\dots,Y_m$ is a basis of $\Ug$ and $\Fg_j=\spa\{Y_1,\dots,Y_j\}$ for $j=1,\dots,m$. 
\end{lemma}

\begin{proof}
	Since $X_1,\dots,X_m$ is a basis of $\Ug$, there exists a unique linear operator 
	$T\colon\Ug\to\Ug$ satisfying $T(X_j)=Y_j$ for $j=1,\dots,m$. 
	The hypothesis is equivalent to $(T-\1)(X_j)\in\Ug_{j-1}$ for $j=1,\dots,m$, 
	and this implies that $T-\1$ is given by a strictly upper triangular matrix with respect to the basis $X_1,\dots,X_m$. 
	Therefore $T$ is invertible, and then the vectors $T(X_j)=Y_j$ for $j=1,\dots,m$ 
	form a basis of~$\Ug$. 
\end{proof}

\begin{proof}[Proof of Theorem~\ref{decomp4}]
	For $j\in e$ we define $Y_j:=X_j$. 
	On the other hand, if $i\in\complement e$, 
	then by the definition of $e$ and
	Proposition~\ref{Gr1}\eqref{Gr1_item12}
	we have 
	$$X_i\in\Wg+\Fg_{i-1}=\Wg\dotplus(\Ug_e\cap\Fg_{i-1})$$ 
	hence there exists a unique vector $Y_i\in\Wg$ with $Y_i-X_i\in\Ug_e\cap\Fg_{i-1}$. 
	Since $Y_i-X_i\in \Ug_e$, we have that 
	$Y_i= X_i-E(\Wg) X_i$.

	We now prove that the set $Y_1,\dots,Y_m$ obtained in this way 
	is a basis of $\Ug$ that 
	also satisfies conditions \eqref{decomp4_item3}--\eqref{decomp4_item4} from the statement. 
	In fact, since $Y_j-X_j\in\Fg_{j-1}$ for $j=1,\dots,m$, it follows by Lemma~\ref{decomp3} 
	that condition~\eqref{decomp4_item3} is satisfied. 
	Moreover, $\{Y_i\mid i\in \complement e\}$ is a linearly independent subset of $\Wg$ 
	whose cardinal is $\vert \complement e\vert=m-\vert e\vert=\dim\Wg$, 
	where the latter equality follows by 
	Proposition~\ref{Gr1}\eqref{Gr1_item10}.  
	Thus condition~\eqref{decomp4_item4} is also satisfied. 
	
	Finally, the map $\beta= (\beta_1, \dots, \beta_m) \colon \Gr_{\Ug_e}(\Ug) \to \Ug^m$, 
	$$ 	\beta_i (\Wg)=
	\begin{cases}
	X_i & \text{if } i\in e, \\
	X_i -E(\Wg) X_i & \text{if } i \in \complement e,
	\end{cases}
	$$
	for $i=1, \dots, m$,  is real analytic by Theorem~\ref{lin0}, and 
	 assertion~\eqref{decomp4_item5} follows directly.
\end{proof}

\section{Nilpotent Lie algebras and groups}\label{nilp}

In this paper, by \textit{nilpotent Lie algebra} we mean a finite-dimensional real vector space $\gg$ endowed with a bilinear map $[\cdot,\cdot]\colon\gg\times\gg\to\gg$ satisfying 
$$[X,Y]=-[Y,X], \ [[X,Y],Z]+[[Y,Z],X]+[[Z,X],Y]=0\text{ and }(\ad_\gg X)^m=0$$ 
for all $X,Y,Z\in\gg$, where $m:=\dim\gg$ and the linear map $\ad_{\gg} X\colon \gg\to\gg$ is defined by $\ad_{\gg}X:=[X,\cdot\,]$. 

The Baker-Campbell-Hausdorff series of $\gg$ 
is defined for arbitrary  $X,Y\in\gg$ by 
\begin{equation}\label{BCH_def_eq1}
X\cdot Y=\sum_{n\ge 1} C_n(X,Y)
\end{equation}
where for $n=1,2,\dots$ we use the notations 
$$C_n(X,Y)=\sum_{k\ge 1} \frac{(-1)^{k-1}}{k} 
\sum_{\stackrel{\scriptstyle p_1+q_1+\cdots+p_k+q_k=n}{\scriptstyle (p_1+q_1)\cdots(p_k+q_k)>0}}
\frac{1}{p_1!q_1!\cdots p_k!q_k!n}\, 
C_{p_1,q_1,\dots,p_k,q_k}(X,Y)
$$
and 
$$C_{p_1,q_1,\dots,p_k,q_k}(X,Y)=
\begin{cases}
(\ad_{\gg}X)^{p_1}(\ad_{\gg}Y)^{q_1}\cdots (\ad_{\gg}X)^{p_k}(\ad_{\gg}Y)^{q_k-1}Y 
&\text{if }q_k\ge1,\\
(\ad_{\gg}X)^{p_1}(\ad_{\gg}Y)^{q_1}\cdots
(\ad_{\gg}X)^{p_k-1}X &\text{if }q_k=0,
\end{cases}$$
whenever $0\le p_1,q_1,\dots,p_k,q_k\in\ZZ$ and $(p_1+q_1)\cdots(p_k+q_k)>0$. 
One has $C_n(X,Y)=0$ whenever $X,Y\in\gg$ and $n\ge \dim\gg$, 
hence the series~\eqref{BCH_def_eq1} actually defines 
a $\gg$-valued polynomial function on $\gg\times\gg$. 
One can check that the formula~\eqref{BCH_def_eq1} defines a group structure on $\gg$, and the corresponding group $G:=(\gg,\cdot)$ will be called here the nilpotent Lie group associated to the nilpotent Lie algebra~$\gg$. 
(See for instance \cite{CG90} and \cite{BB15a} for more details.)

Equivalently, one can define a nilpotent Lie group as a pair $G=(\gg,\cdot)$ consisting of a finite-dimensional real vector space $\gg$ and a group structure $\gg\times\gg\to\gg$, $(X,Y)\mapsto X\cdot Y$, which is a polynomial map and satisfyies $(tX)\cdot (sX)=(t+s)X$ for all $t,s\in\RR$ and $X\in\gg$. 
Defining 
\begin{equation}
\label{BCH_def_eq2}
[X,Y]:=\frac{\partial^2}{\partial t\partial s}\Big\vert_{t=s=0}(tX)\cdot(sY)\cdot(-tX)
\end{equation}
one obtains a map $[\cdot,\cdot]\colon\gg\times\gg\to\gg$ that turns $\gg$ into a nilpotent Lie algebra whose corresponding Lie group is $G=(\gg,\cdot)$ and \eqref{BCH_def_eq1} holds true. 

\begin{remark}
\normalfont
It is clear from \eqref{BCH_def_eq1}--\eqref{BCH_def_eq2} that 
$G=(\gg,\cdot)$ is an abelian group if and only if 
$[\cdot,\cdot]=0$ (and then we say that $\gg$ is an abelian Lie algebra), 
and this is further equivalent to $X\cdot Y=X+Y$ for all $X,Y\in\gg$. 
If this is the case, then $\gg$ and $G$ are nothing else than a finite-dimensional real vector space. 
For this reason we regard the nilpotent Lie algebras and groups as noncommutative generalizations of vector spaces. 
\end{remark}

\subsection*{Application of Theorem~\ref{decomp4} to nilpotent Lie algebras}

\begin{remark}\label{decomp5}
\normalfont
Let $\gg$ be a nilpotent Lie algebra. 
We denote by $\JH(\gg)$ the set of all Jordan-H\"older bases of $\gg$, 
that is, the bases $(X_1,\dots,X_m)\in\gg^m$ for which, denoting 
$\gg_k:=\spa\{X_j\mid 1\le j\le k\}$, one has $[\gg,\gg_k]\subseteq\gg_{k-1}$ for $k=1,\dots,m$, where $\gg_0:=\{0\}$. 

Let $(X_1,\dots,X_m)\in\JH(\gg)$, 
and $a=\{i_1,\dots,i_s\}\subseteq\{1,\dots,m\}$ with $i_1<\cdots<i_s$.  
If $\gg_a:=\spa\{X_i\mid i\in a\}$ is a subalgebra of $\gg$, then 
$[\gg_a,\gg_a]\subseteq\gg_a$, and then it easily follows that 
$X_{i_1},\dots,X_{i_s}$ is a Jordan-H\"older basis of~$\gg_a$. 
\end{remark}

\begin{corollary}\label{decomp6}
Let $\gg$ be a nilpotent Lie algebra    
and $(X_1,\dots,X_m)\in\JH(\gg)$.   
Then for every $e\subseteq\{1,\dots,m\}$  there exists a unique map 
 $\beta\colon \Gr_{{\Fg_\bullet},e}(\gg)\to\gg^m$
satisfying the following conditions for every $\hg\in\Gr_{{\Fg_\bullet},e}(\gg)$, 
with $\beta(\hg)=:(Y_1,\dots,Y_m)\in\gg^m $: 
\begin{enumerate}[{\rm(i)}]
	\item\label{decomp6_item1} 
	If $j\in e$ then $Y_j=X_j$. 
	\item\label{decomp6_item2} 
	If $i\in\complement e$ then $Y_i-X_i\in\gg_e\cap\Fg_{i-1}$. 
\end{enumerate}
Moreover, $\beta(\hg)$ is a basis of $\gg$ having the following properties: 
\begin{enumerate}[{\rm(i)}]
	\setcounter{enumi}{2}
	\item\label{decomp6_item3} 
	One has $\Fg_j=\spa\{Y_1,\dots,Y_j\}$ for $j=1,\dots,m$. 
	\item\label{decomp6_item4} 
	The set $\{Y_i\mid i\in \complement e\}$ is a basis of $\hg$. 
	\item\label{decomp6_item5}  The mapping 
	$\beta\colon \Gr_{{\Fg_\bullet},e}(\gg)\to\gg^m$, $\Wg\mapsto\beta(\Wg)$,
	can be extended to a real analytic mapping on the open subset $\Gr_{\gg_e}(\gg)$ of $\Gr(\gg)$.
\item\label{decomp6_item6} One has $\beta(\hg)\in\JH(\gg)$. 
\item\label{decomp6_item7} If $\hg$ is a subalgebra of $\gg$ then $(Y_{i_1},\dots,Y_{i_s})\in\JH(\hg)$, 
where we have denoted 
$\complement e=:\{i_1<\cdots<i_s\}$. 
\end{enumerate}
\end{corollary}

\begin{proof}
Assertions \eqref{decomp6_item1}--\eqref{decomp6_item5} follow by Theorem~\ref{decomp4} for $\Ug=\gg$ and $\Wg=\hg$. 
	
Since $(X_1,\dots,X_m)\in\JH(\gg)$, it follows by \eqref{decomp4_item3} that 
$(Y_1,\dots,Y_m)\in\JH(\gg)$. 

Moreover, we have $\hg=\spa\{Y_{i_1},\dots,Y_{i_s}\}$ by Theorem~\ref{decomp4}\eqref{decomp4_item4}, 
hence Remark~\ref{decomp5} shows that Assertion~\eqref{decomp6_item7} holds true. 
\end{proof}

\begin{remark}
\normalfont 
The above Corollary~\ref{decomp6} 
 contains some results from \cite{Co98}, 
which are sufficiently general and precise for the applications we wish to make in this paper. 
\end{remark}

\section{Nonlinear oblique projections in nilpotent Lie algebras}\label{nonlin}

In this section we establish our main result on nonlinear oblique projections (Theorem~\ref{of}). 
Here $\gg$ is a nilpotent Lie algebra with a fixed Jordan-H\"older sequence
\begin{equation}\label{JH-seq}
\Fg_\bullet:\quad \{0\}=\Fg_0\subseteq\Fg_1\subseteq\cdots\subseteq\Fg_m=\gg.
\end{equation}
We denote 
$$\JH_{\Fg_\bullet}(\gg):=(\Fg_1\setminus\Fg_0)
\times(\Fg_2\setminus\Fg_1)\times\cdots\times(\Fg_m\setminus\Fg_{m-1})\subseteq(\gg\setminus\{0\})^m.$$
It is easily seen that $\JH_{\Fg_\bullet}(\gg)$ is exactly the set of all Jordan-H\"older bases of $\gg$ 
which are compatible with the above Jordan-H\"older sequence, that is, the $m$-tuples $\underline{X}=(X_1,\dots,X_m)$ 
satisfying $\Fg_k=\spa\{X_j\mid 1\le j\le k\}$ for $k=1,\dots,m$. 
See also Lemma~\ref{decomp3}.

The following lemma is a generalization of a known result. 
(Compare for instance \cite[Prop. 1.2.7]{CG90} and \cite[Prop. 5.2.6]{FuLu15}.)  
The point here is that we do not only establish factorizations that 
involve arbitrary partitions of Jordan-H\"older bases, but we  also take into account dependence 
on bases that are compatible with a fixed Jordan-H\"older sequence. 
We need these enhanced features in order to obtain analyticity of nonlinear oblique projections in Theorem~\ref{of}. 

\begin{lemma}\label{block}
Fix any partition $\{1,\dots,m\}=A_1\sqcup\cdots\sqcup A_k$ and define the polynomial map 
$$\Phi\colon\RR^m\times \gg^m \to\gg,\quad 
\Phi(t_1,\dots,t_m,\underline{X}):=\Bigl(\sum_{j\in A_1}t_j X_j\Bigr)\cdots\Bigl(\sum_{j\in A_k}t_j X_j\Bigr).$$
Then for $j=1,\dots,m$ there exists a polynomial function 
$P_j\colon \RR^{m-j}\times \JH_{\Fg_\bullet}(\gg)\to\RR$
with  
$$
\Phi(t_1,\dots,t_m,\underline{X})=\sum_{j=1}^m(t_j+P_j(t_{j+1},\dots,t_m,\underline{X}))X_j$$
for all $t_1,\dots,t_m\in\RR$ and $\underline{X}=(X_1,\dots,X_m)\in\JH_{\Fg_\bullet}(\gg)$. 
Moreover, for every $\underline{X}\in\JH_{\Fg_\bullet}(\gg)$, 
the map $\Phi(\cdot,\underline{X})$ is a polynomial diffeomorphism $\RR^m\to\gg$, 
and its inverse map defines a polynomial function 
$\gg\times \JH_{\Fg_\bullet}(\gg)\to\RR^m$, $(Y,\underline{X})\mapsto \Phi(\cdot,\underline{X})^{-1}(Y)$. 
\end{lemma}

Here we use the convention that if $A_r=\emptyset$ then 
$\sum\limits_{j\in A_r}t_j X_j:=0\in\gg$. 
We also note that the polynomial function $P_m$ is constant. 

\begin{proof}
We proceed by induction on $m$. 
The case $m=1$ is clear. 
Let us assume that the assertion holds true for all nilpotent Lie algebras of dimension~$<m$. 

Let $k_1\in\{1,\dots,k\}$ with $1\in A_{k_1}$. 
Define $\widetilde{A_j}:=A_j$ if $j\in\{1,\dots,k\}\setminus\{k_1\}$ 
and $\widetilde{A_{k_1}}:=A_{k_1}\setminus\{1\}$. 

Since $(X_1,\dots,X_m)\in\JH(\gg)$, it follows that $\gg_1:=\RR X_1$ is contained in the center of~$\gg$. 
Define $\widetilde{\gg}:=\gg/\gg_1$ with its Jordan-H\"older sequence 
$$\widetilde{\Fg}_\bullet:\quad\{0\}\subseteq\gg_2/\gg_1\subseteq\cdots\subseteq\gg_m/\gg_1=\widetilde{\gg}$$ 
and $\pi\colon\gg\to\widetilde{\gg}$, $\pi(X):=X+\gg_1$. 
Then $\widetilde{X}:=(\pi(X_2),\dots,\pi(X_m))
\in\JH_{\widetilde{\Fg}_\bullet}(\widetilde{\gg})$ 
and one has the partition $\{2,\dots,m\}=\widetilde{A_1}\sqcup\cdots\sqcup \widetilde{A_k}$, 
hence we can define the corresponding map 
$\widetilde{\Phi}\colon\RR^{m-1}\times \widetilde{\gg}^{m-1}\to\widetilde{\gg}$ 
by 
$$\widetilde{\Phi}(t_2,\dots,t_m,\underline{Y}):=\Bigl(\sum_{j\in A_1}t_j Y_j\Bigr)\cdots
\Bigl(\sum_{j\in A_{k_1}\setminus\{1\}}t_j Y_j\Bigr)\cdots
\Bigl(\sum_{j\in A_k}t_j Y_j\Bigr), $$
for $\underline{Y}=(Y_2, \dots, Y_m)\in \widetilde{\gg}^{m-1}$.
It follows by the induction hypothesis that 
\begin{equation}\label{block_proof_eq1}
\widetilde{\Phi}(t_2,\dots,t_m,\widetilde{X})=\sum_{j=2}^m(t_j+P_j(t_{j+1},\dots,t_m,\widetilde{X}))\pi(X_j)
\end{equation}
for suitable polynomial functions $P_j\colon\RR^{m-j}\times\JH_{\widetilde{\Fg}_\bullet}(\widetilde{\gg})\to\RR$. 

On the other hand, 
denoting 
$$ \Psi(t_2,\dots,t_m,\underline{X}):=\Bigl(\sum_{j\in A_1}t_j X_j\Bigr)\cdots
\Bigl(\sum_{j\in A_{k_1}\setminus\{1\}}t_j X_j\Bigr)\cdots
\Bigl(\sum_{j\in A_k}t_j X_j\Bigr)$$
one has 
$$\widetilde{\Phi}(t_2,\dots,t_m,\widetilde{X})=\pi\Bigl(\Psi(t_2,\dots,t_m,\underline{X})\Bigr) $$
and it follows by \eqref{block_proof_eq1} that there exists a polynomial $P_1\colon\RR^{m-1}\times\JH_{\Fg_\bullet}(\gg)\to\RR$ with 
$$\Psi(t_2,\dots,t_m,\underline{X})=P_1(t_2,\dots,t_m,\underline{X})X_1+
\sum_{j=2}^m(t_j+P_j(t_{j+1},\dots,t_m,\widetilde{X}))X_j.$$
Now recall that $X_1$ belongs to the center of $\gg$, 
hence $X_1\cdot Y=X_1+Y$ for every $Y\in\gg$. 
This implies that 
$$\begin{aligned}
\Phi(t_1,\dots,t_m,\underline{X})
&=t_1X_1+\Psi(t_2,\dots,t_m,\underline{X}) \\
&=(t_1+P_1(t_2,\dots,t_m,\underline{X}))X_1+
\sum_{j=2}^m(t_j+P_j(t_{j+1},\dots,t_m,\widetilde{X}))X_j,
\end{aligned}$$
which completes the induction step. 

Using the formula thus established for $\Phi$, it is straightforward to prove that 
$\Phi$ is a polynomial diffeomorphism whose inverse map is also polynomial, 
and this completes the proof. 
\end{proof}

\begin{remark}
\normalfont
The set  $\Gr^{\alg}(\gg)$ of all subalgebras of $\gg$ is a Zariski closed subset of the manifold $\Gr(\gg)$.
If $\dim\gg=m$, then we define 
$$\Gr^{\alg}(\gg,k):=\Gr^{\alg}(\gg)\cap\Gr(\gg,k)
\text{ for }k=1,\dots,m.$$
For any Jordan-H\"older sequence \eqref{JH-seq} and any $e\subseteq\{1,\dots,m\}$ we also define 
$$\Gr_{{\Fg_\bullet},e}^{\alg}(\gg)
:=\Gr^{\alg}(\gg)\cap\Gr_{{\Fg_\bullet},e}(\gg)
=\{\hg\in\Gr^{\alg}(\gg)\mid\jump_{\Fg_\bullet}(\hg)=e\}.$$
\end{remark}

\begin{theorem}\label{of}
Let $(V_1,\dots,V_m)\in\JH_{\Fg_\bullet}(\gg)$, 
$e\subseteq\{1,\dots,m\}$ be fixed and 
and define $\gg_e:=\spa\{V_j\mid j\in e\}$. 
For every $\hg\in\Gr_{{\Fg_\bullet},e}^{\alg}(\gg)$ the Baker-Campbell-Hausdorff multiplication defines a polynomial diffeomorphism 
$\gg_e\times\hg\to\gg$. 
Moreover,
there exist a real analytic map $\beta= (Y_1(\cdot), \dots, Y_m(\cdot))\colon \Gr_{\gg_e}(\gg)\to \gg^m 
$ and a polynomial map $p=(p_1, \dots, p_m)\colon \gg\times \JH_{\Fg_\bullet}(\gg) \to \RR^m$, such that
\begin{enumerate}[{\rm (i)}]
	\item For  $\hg\in\Gr_{{\Fg_\bullet},e}^{\alg}(\gg)$ one has $\beta(\hg)=(Y_1(\hg),\dots,Y_m(\hg))\in\JH_{\Fg_\bullet}(\gg)$.
	\item
	The function $\Pi\colon \gg\times\Gr_{\gg_e}(\gg)\to\gg_e$ given by 
$$\Pi\colon \gg\times\Gr_{{\Fg_\bullet},e}^{\alg}(\gg)\to\gg_e,\quad \Pi(Y,\hg)=\sum_{j\in e}p_j(Y,\beta(\hg))Y_j(\hg)$$
has the property that
$$(\forall X\in\gg)(\forall \hg\in\Gr_{{\Fg_\bullet},e}^{\alg}(\gg))\quad X\in\Pi(X,\hg)\cdot\hg, $$
and it is uniquely determined.
\end{enumerate}
\end{theorem}

\begin{proof}
For arbitrary $\hg\in\Gr_{\gg_e}(\gg)$ we define $\beta(\hg)=(Y_1(\hg),\dots,Y_m(\hg))\in\gg^m$ via Corollary~\ref{decomp6},  
hence for  $\hg\in\Gr_{{\Fg_\bullet},e}^{\alg}(\gg)$, 
$\beta(\hg)\in\JH_{\Fg_\bullet}(\gg)$  and $\hg=\spa\{Y_i(\hg)\mid i\in\complement e\}$ and $\gg_e=\spa\{Y_j(\hg)\mid j\in e\}$ 

By Lemma~\ref{block} applied for the partition $\{1,\dots,m\}=e\sqcup\complement e$, we obtain that the map 
$$\Phi\colon\RR^m\times \JH_{\Fg_\bullet}(\gg)\to\gg,\quad 
\Phi(t_1,\dots,t_m,\underline{X}):=
\Bigl(\sum_{j\in e}t_j X_j\Bigr)\cdot\Bigl(\sum_{i\in \complement e}t_i X_i\Bigr)$$
has the property that 
for every $\underline{X}=(X_1,\dots,X_m)\in\JH_{\Fg_\bullet}(\gg)$, 
the map $\Phi(\cdot,\underline{X})$ is a polynomial diffeomorphism $\RR^m\to\gg$ 
whose inverse map gives a polynomial function 
\begin{equation}\label{of_proof_p}
p\colon \gg\times \JH_{\Fg_\bullet}(\gg)\to\RR^m, \quad 
(Y,\underline{X})\mapsto \Phi(\cdot,\underline{X})^{-1}(Y)=(p_1(Y,\underline{X}),\dots,p_m(Y,\underline{X}))
\end{equation}
hence one has the unique factorization 
\begin{equation}\label{of_proof_eq1}
Y=\Bigl(\sum_{j\in e}p_j(Y,\underline{X}) X_j\Bigr)
\cdot\Bigl(\sum_{i\in \complement e}p_i(Y,\underline{X}) X_i\Bigr)
\end{equation}
for all $Y\in\gg$ and $\underline{X}=(X_1,\dots,X_m)\in\JH(\gg)$. 

For $\hg\in\Gr_{{\Fg_\bullet},e}^{\alg}(\gg)$ and $\underline{X}=\beta(\hg)$ it thus follows 
that the Baker-Campbell-Hausdorff multiplication defines a polynomial diffeomorphism 
$\gg_e\times\hg\to\gg$. 

Using the polynomial map $p$ in \eqref{of_proof_p},
we now define the mapping
$$\Pi\colon \gg\times\Gr_{{\Fg_\bullet},e}^{\alg}(\gg)\to\gg_e,\quad \Pi(Y,\hg)=\sum_{j\in e}p_j(Y,\beta(\hg))Y_j(\hg),$$
 which satifies the condition in the statement.
Using the uniqueness of the factorization~\eqref{of_proof_eq1} and the equality $\hg=\spa\{Y_i(\hg)\mid i\in\complement e\}$, 
it follows that $\Pi(X,\hg)\in\gg_e$ is uniquely determined by the condition $X\in \Pi(X,\hg)\cdot\hg$, 
and this completes the proof. 
\end{proof}

\begin{remark}
	\normalfont
	With the notations and in the contitions of Theorem~\ref{of}, we have in fact obtained that for every $X\in \gg$, and $\hg\in\Gr_{{\Fg_\bullet},e}^{\alg}(\gg))$, $X=\Pi(X,\hg)$ $(\text{mod } \hg)$, where on $\gg$ we consider the nilpotent Lie group structure given by the Baker-Campbell-Hausdorff multiplication, such that $\hg$ becomes a subgroup of $\gg$.
	\end{remark}

\begin{example}
	\label{rank2}
	\normalfont
Let $\gg$ be a nilpotent Lie algebra with $\dim(\gg/[\gg,\gg])=2$, 
and assume that 
$${\Fg_\bullet}:\quad \{0\}=\Fg_0\subset\Fg_1\subset\cdots\subset\Fg_m=\gg$$
is a Jordan-H\"older sequence with $\Fg_{m-2}=[\gg,\gg]$. 
This implies that 
\begin{equation}
\label{rank2_eq1}
\Gr^{\alg}(\gg,m-1)=\{\hg\in\Gr(\gg,m-1)\mid \Fg_{m-2}\subset\hg\},
\end{equation}
using \cite[Lemma 1.1.8]{CG90}. 
Consequently, the map 
$$\Gr^{\alg}(\gg,m-1)\to\Gr(\gg/\Fg_{m-2},1),\quad 
\hg\mapsto \hg/\Fg_{m-2}$$
is a bijection. 
Since $\dim(\gg/\Fg_{m-2})=2$, it thus follows that 
$\Gr^{\alg}(\gg,m-1)$ is homeomorphic to the real projective line (i.e., the space of all 1-dimensional subspaces of a 2-dimensional real vector space), which is further homeomorphic to the unit circle~$\TT$. 
A more specific parameterization of $\Gr^{\alg}(\gg,m-1)$ is the 2-sheeted covering map 
$$\TT\to\Gr^{\alg}(\gg,m-1),\quad z\mapsto \hg_z$$
where we define 
$$\hg_z:=\spa(\Fg_{m-2}\cup\{(\cos\theta) X_{m-1}+(\sin\theta) X_m\}) 
\text{ for }z=\ee^{\ie\theta}\in\TT.$$ 
For any $\hg\in\Gr^{\alg}(\gg,m-1)$ one has by \eqref{rank2_eq1} along with Examples \ref{ex1}~and~\ref{ex_codim1}, 
\begin{itemize} 
	\item either $\Fg_{m-1}=\hg$, i.e., $\hg=\hg_z$ with $z=1$, 
	and then $\jump_{\Fg_\bullet}(\hg)=\{m\}$;
	\item or $\Fg_{m-1}\not\subseteq\hg$, i.e., $\hg=\hg_z$ with $z\in\TT\setminus\{1\}$, 
	and then $\jump_{\Fg_\bullet}(\hg)=\{m-1\}$. 
\end{itemize}
Thus 
$$\Gr_{{\Fg_\bullet},\{m\}}^{\alg}(\gg)=\{\hg_1\}
\text{ and }
\Gr_{{\Fg_\bullet},\{m-1\}}^{\alg}(\gg)=\{\hg_z\mid z\in\TT\setminus\{1\}\}.$$
\end{example}

We will now specialize Example~\ref{rank2} to two nilpotent Lie algebras whose coadjoint orbits have dimensions less than~2, 
which were classified in \cite{ACL95}. 

\begin{example}
	\label{filiform}
\normalfont
For an arbitrary integer $m\ge 3$ let $\gg$ be the $m$-dimensional threadlike Lie algebra, that is, 
the nilpotent Lie algebra with a basis 
$X_1,\dots,X_m$ satisfying the commutation relations 
$$[X_m,X_j]=X_{j-1}\text{ for }j=2,\dots,m-1$$ 
and $[X_k,X_j]=0$ if 
$1\le j< k\le m-2$. 

We define $\Fg_k:=\spa\{X_j\mid 1\le j\le k\}$ for $k=1,\dots,m$ and $\Fg_0:=\{0\}$. 
The center of $\gg$ is $\Fg_1=\RR X_1$, 
while $[\gg,\gg]=\Fg_{m-2}$, and thus Example~\ref{rank2} applies and we will use its notation.  

It is easily checked that $\hg_z$ is isomorphic to the $(m-1)$-dimensional threadlike Lie algebra if $z\in\TT\setminus\{1\}$. 
Moreover, one has 
$$[\hg_z,\hg_z]=
\begin{cases}
\{0\}&\text{ if }z=1,\\
\Fg_{m-3}&\text{ if }z\in\TT\setminus\{1\}.
\end{cases}$$
Therefore for every $z\in\TT$ the subalgebra $\hg_z$ is subordinated to any $\xi\in\Fg_{m-3}^\perp\subset\gg^*$. 
\end{example}

\begin{example}
	\label{N5N4}
	\normalfont 
	Let $\gg$ be the nilpotent Lie algebra 
	with a basis $X_1,X_2,X_3,X_4,X_5$ satisfying the commutation relations
$$[X_5,X_4]=X_3,\ [X_5,X_3]=X_2,\ [X_4,X_3]=X_1.$$
We define $\Fg_k:=\spa\{X_j\mid 1\le j\le k\}$ for $k=1,\dots,5$ and $\Fg_0:=\{0\}$. 

It is clear that 
the center of $\gg$ is $\Fg_2=\spa\{X_1,X_2\}$. 
Moreover $[\gg,\gg]=\Fg_3=\spa\{X_1,X_2,X_3\}$, 
and thus Example~\ref{rank2} applies (with $m=5$) and we will use its notation.  

For $z=\ee^{\ie\theta}\in\TT$, denoting $X(z):=(\cos\theta) X_1+(\sin\theta) X_2\in\Fg_2$, one has 
$$[(\cos\theta) X_4+(\sin\theta) X_5,X_3]
=(\cos\theta) X_1+(\sin\theta) X_2=X(z)$$
and then it follows that $\hg_z$ is isomorphic to the direct product of a 3-dimensional Heisenberg algebra and a 1-dimensional Lie algebra. 
In particular, $\hg_{z_1}$ is isomorphic to $\hg_{z_2}$ for all $z_1,z_2\in\TT$, unlike Example~\ref{filiform}. 
Moreover, one has 
$$[\hg_z,\hg_z]=\RR X(z) 
\text{ for all }z\in\TT.$$
Therefore, if $z\in\TT$, then the subalgebra $\hg_z$ is subordinated to $\xi\in\gg^*$ if and only if $X(z)\in\Ker\xi$. 
\end{example}

\begin{remark}
\normalfont
Theorem~\ref{of} sheds fresh light on the topology of the dual space of a nilpotent Lie group~$G$. 
In fact, the coadjoint isotropy and the Vergne polarization with respect to a fixed Jordan-H\"older sequence 
define some maps $\gg^*\to\Gr_{\alg}(\gg)$, and the corresponding preimages of Schubert cells, 
when factorized through the coadjoint action, correspond to subquotients of the $C^*$-algebra of $G$ 
that have remarkable properties 
as for instance continuous trace or Morita equivalence to commutative $C^*$-algebras. 
See \cite{BBL17}
for more details.  
\end{remark}


\end{document}